\documentclass[12pt]{amsart}
\usepackage{a4wide,enumerate,color,mathtools}
\allowdisplaybreaks

\usepackage[capitalise]{cleveref}

\let\pa\partial

\let\var\varepsilon

\newcommand{\R}{{\mathbb R}}
\newcommand{\T}{{\mathbb T}}

\newcommand{\indicator}{\mathbb{I}}

\newtheorem{theorem}{Theorem}
\newtheorem{lemma}[theorem]{Lemma}
\newtheorem{proposition}[theorem]{Proposition}
\newtheorem{remark}[theorem]{Remark}

\newtheorem{definition}{Definition}

\newcommand{\floor}[1]{\lfloor #1 \rfloor}
\newcommand{\fullV}[1]{\underline{\underline{#1}}} % x in full space
\newcommand{\projV}[1]{\underline{#1}} % projected x
\newcommand{\probP}{\mathcal{P}} % Probability P
\newcommand{\projP}{\mathbb{P}} % Projection P
\newcommand{\dd}{\mathrm{d}} % constant d for differential
\newcommand{\entropyH}{{\mathcal H}} % Entropy functional H

\begin{document}

\title[Macroscopic limit]{About the entropic structure of detailed balanced multi-species cross-diffusion equations}

\author[E. S. Daus]{Esther S. Daus}
\address{Institute for Analysis and Scientific Computing, Vienna University of
	Technology, Wiedner Hauptstra\ss e 8--10, 1040 Wien, Austria}
\email{esther.daus@tuwien.ac.at}

\author[L. Desvillettes]{Laurent Desvillettes}
\address{Universit\'e Paris Diderot, Sorbonne Universit\'e, CNRS, Institut de Math\'ematiques de Jussieu-Paris Rive Gauche, IMJ-PRG, F-75013, Paris, France.}
\email{desvillettes@math.univ-paris-diderot.fr}

\author[H. Dietert]{Helge Dietert}
\address{Universit\'e Paris Diderot, Sorbonne Universit\'e, CNRS, Institut de Math\'ematiques de Jussieu-Paris Rive Gauche, IMJ-PRG, F-75013, Paris, France.}
\email{dietert@math.univ-paris-diderot.fr}

\date{\today}

\thanks{The first author acknowledges partial support from the
  Austrian Science Fund (FWF), grants P27352 and P30000. The
  research leading to this paper was also partially funded by the
  French ``ANR blanche'' project Kibord: ANR-13-BS01-0004, by the GDRI
  of the CNRS ReaDiNet, Reaction-Diffusion Network in Biomedecine, and
  by Universit\'e Sorbonne Paris Cit\'e, in the framework of the
  ``Investissements d'Avenir'', convention ANR-11-IDEX-0005. The last
  author acknowledges partial support from the People Programme (Marie
  Curie Actions) of the European Union’s Seventh Framework Programme
  (FP7/2007-2013) under REA grant agreement n.
  PCOFUND-GA-2013-609102, through the PRESTIGE programme coordinated
  by Campus France.}

\begin{abstract}
  This paper links at the formal level the entropy structure of a
  multi-species cross-diffusion system of Shigesada-Kawasaki-Teramoto
  (SKT) type (cf. \cite{SKT79}) satisfying the detailed balance
  condition with the entropy structure of a reversible microscopic
  many-particle Markov process on a discretised space. The link is
  established by first performing a mean-field limit to a master
  equation over discretised space. Then the spatial discretisation
  limit is performed in a completely rigorous way. This by itself
  provides a novel strategy for proving global existence of weak
  solutions to a class of cross-diffusion systems.
\end{abstract}



\maketitle

\section{Introduction}
\label{sec:introduction}

We consider the population dynamics cross-diffusion system model
coming out of the classical paper by Shigesada, Kawasaki and Teramoto
\cite{SKT79} (SKT model) for $n\geq 2$ species without reaction
term. For clarity, we suppose that the species live on the torus
$\T = [0,1)$ with periodic boundary conditions. Thus, the density
$u_i := u_i(t,x)$ of species $i=1,\dots ,n$ evolves as
\begin{equation}
  \label{eq:intro-skt}
  \partial_t u_i = \Delta \left(D_i u_i + \sum_{j=1}^{n} A_{ij} u_j u_i \right)
\end{equation}
with diffusion constants $D_i \geq 0$, self-diffusion coefficients $A_{ii}>0$ and cross-diffusion coefficients $A_{ij}\geq 0$ for $i,j=1,\dots,n$.

For this system \eqref{eq:intro-skt}, Chen, Daus and J\"{u}ngel showed in \cite{CDJ18} that
\begin{equation}
  \label{eq:entropy-skt}
  \entropyH (u) := \int_{\T} \sum_{i=1}^{n}
  \pi_i \left[
    u_i \log (u_i(x)) - u_i(x) + 1
  \right] \dd x
\end{equation}
with positive constants $\pi_i>0$ for $i=1,\dots ,n$ is an entropy (Lyapunov)
functional if the following condition holds
\begin{equation}
  \label{eq:detailed-balance-skt}
  \pi_i A_{ij} = \pi_j A_{ji}
  \qquad \text{for $i,j=1,\dots,n$,}
\end{equation}
which for $n \geq 3$ gives a constraint on the cross-diffusion
coefficients $A_{ij}$.  Under this condition (they called it
\textit{detailed balance condition}), the authors were then able to
construct global weak solutions to \eqref{eq:intro-skt} for an
arbitrary number of population species with the help of the gradient
estimates coming from the entropy production of the entropy
\eqref{eq:entropy-skt}.

The motivation of this work is to understand the origin of the entropy
\eqref{eq:entropy-skt} under the condition
\eqref{eq:detailed-balance-skt}. In particular, we wanted to link the
condition \eqref{eq:detailed-balance-skt} to the detailed balance
equation of finite-state Markov chains, where the detailed balance
equation has been identified as necessary and sufficient condition for
the existence of a gradient flow structure with respect to the
relative entropy
\cite{dietert-2015-characterisation,maas-2011-gradient-flows,mielke-2011-gradient}.

In this work, we establish the formal link between the entropy
structure of \eqref{eq:intro-skt} and the entropy structure of a
microscopic many-particle Markov process on a discrete space. The link
is established in two steps. In the first step, we perform a formal
mean-field limit keeping the spatial discretisation fixed. The
resulting system is a quadratic master equation. In the second step,
we then refine the spatial discretisation and arrive at the cross-diffusion system
\eqref{eq:intro-skt}.

In this many-particle derivation, the condition
\eqref{eq:detailed-balance-skt} enters as a natural necessary condition
for the construction of a reversible Markov process and the
constants $\pi_i$ can be interpreted as relative portions in the
many-particle model.

In both limits, the entropy structure is preserved and, in particular,
the master equation on the discretised space has the corresponding
entropy structure. This allows us to perform the spatial
discretisation limit in a rigorous way, which is an interesting result
by itself (which will be discussed after the statement of our main
Theorem \ref{main} in \cref{sec:rigorous.derivation}).

Note that the transfer of the entropy structure from a microscopic model towards a mesoscopic model has been extensively studied for equations belonging to other classes. The spatially homogeneous Boltzmann equation is for example a model in which many results have been proven (cf. \cite{MM13}). It shares some features with the SKT model (quadraticity of course, but also diffusive properties when the angular cutoff of Grad is not performed).

For the second rigorous limit from the space discretised master
equation to the SKT cross-diffusion system, similar discrete in space
approximation schemes for the SKT model were studied in \cite{GLS09,
  ABR11, Mur17}, but they are not necessarily entropy preserving. Very
recently, an entropy preserving numerical scheme was proposed in
\cite{CCGJ18}, though not for the SKT model, but for a volume-filling
type cross-diffusion system.

Other approaches have been proposed for obtaining cross-diffusion
equations of SKT type out of microscopic models. First (stochastic) approaches from particle models to reaction-diffusion systems trace back to Oelschl\"{a}ger \cite{Oel89} in the late 1980s. Recently, Fontbona and M\'el\'eard \cite{FM15} managed to prove the convergence from realistic individual-based models in a suitable
limit towards non-local (convoluted w.r.t. space) SKT-type
systems. Note that because of the lacking evidence of existence and
uniqueness of strong solutions to the limiting model, it looks
difficult to provide a rigorous proof of passage to the limit towards
the full (i.e. nontriangular) local multi-species SKT system when one
starts with a microscopic model (whether on a discrete set of
positions or on a continuous set of positions, using a nonlocality
which disappears in the limit). Note, however, that very recently, Moussa [21] manged to prove the convergence in the
  case of strictly triangular limiting local SKT model with bounded
  coefficients starting on a continuous set of positions with a
  nonlocality which disappears in the limit by using duality
  techniques (introduced for instance by Pierre and Schmitt in \cite{PS97}).
  %Our results presented here follow a
  %different approach, namely to first pass formally from a particle
  %model with the help of the BBGKY hierarchy under the molecular chaos
  %assumption to a discrete in space master equation, from where we are
  %able to pass rigorously to the full (i.e. nontriangular) multi-species
  %SKT model.

%%%%%%%%%%%%%%%%%%%%%%%%%%%%%%%%%%%%%%%%%%%%%%%%%%%%%%%%%%%%%%%%%%%%%%%%%%%%%%%
\section{Formal mean-field limit}
For the microscopic derivation, we first consider a many-particle
system on a fixed spatial discretisation. The spatial discretisation
consists of $M$ positions given by
\begin{equation}
  \label{eq:def-omega-m}
  \Omega_M =\{x_k : k=0,\dots,M-1\}
  \qquad\text{with}\qquad
  x_k = \frac{k}{M} = kh,
\end{equation}
which is understood in the periodic setting, and where  we set $h=M^{-1}$.

Given the relative fractions $\pi_1,\dots,\pi_n$ with
  $\pi_i >0$ between the species, we consider the
many-particle system with $\floor{\pi_i N}$ particles
of species $i=1,\dots,n$, where $\floor{\pi_i N}$
denotes the largest integer smaller than $\pi_i N$. The aim of this
section is to obtain a suitable master equation when $N\to\infty$.

The microscopic configuration is given by
\begin{equation*}
  \fullV{x} := (x_1^1,\dots,x_1^{\floor{\pi_1 N}},x_2^1,\dots,x_2^{\floor{\pi_2
    N}},\dots \dots,x_n^1,\dots,x_n^{\floor{\pi_n N}})
  \in \Omega_M^{\otimes \left(\floor{\pi_1N}+\dots+\floor{\pi_nN}\right)} =: \Omega_M^N
\end{equation*}
and this configuration is set to evolve in time as a time-continuous
Markov chain.

The distribution over the microscopic configurations at time $t$ is
given by a density $\mu^N_t \in \probP(\Omega^N_M)$. In terms of
statistical physics, this means that we consider an ensemble over the
microscopic configurations.

We assume that the particles within a species are
indistinguishable. The class of such measures is denoted by
$\probP_s(\Omega_M^N)$ and defined as follows:

\begin{definition}[Indistinguishability]
  A measure $\mu \in \probP(\Omega^N_M)$ is in $\probP_s(\Omega^N_M)$
  if and only if for all permutations $\sigma_1$,\dots,$\sigma_n$ of
  resp. $\{1,..,\floor{\pi_1N}\}$, \dots, $\{1,..,\floor{\pi_nN}\}$ and configurations
  $\fullV{x} \in \Omega_M^N$ it holds that
  \begin{equation*}
    \begin{aligned}
      &\mu^N(x_1^{\sigma_1(1)},\dots,x_1^{\sigma_1(\floor{\pi_1 N})},
      x_2^{\sigma_2(1)},\dots,x_2^{\sigma_2(\floor{\pi_2 N})},
      \dots \dots, x_n^{\sigma_n(1)},\dots,x_n^{\sigma_n(\floor{\pi_n N})}) \\
      &= \mu^N(x_1^1,\dots,x_1^{\floor{\pi_1 N}},x_2^1,\dots,x_2^{\floor{\pi_2
        N}},\dots \dots,x_n^1,\dots,x_n^{\floor{\pi_n N}}).
    \end{aligned}
  \end{equation*}
\end{definition}

By the indistinguishability, the distribution of a typical particle is
given by the marginal distribution. For this, we first introduce the
following notation for projections.

\begin{definition}[Projections]
  Let $p=(p_1,\dots,p_n)$ and $N$ be such that
  $p_i \le \floor{\pi_i N}$ for $i=1,\dots,n$.  We define the
  projection
  \begin{equation*}
    \projP^{N;(p)} : \probP(\Omega_M^N)
    \mapsto \probP(\Omega_M^{\otimes (p_1+\dots +p_n)})
  \end{equation*}
  by
  \begin{equation*}
    \begin{aligned}
      (\projP^{N;(p)} \mu^N)(\projV{x})
      := \sum_{x_1^{p_1+1}\in \Omega_M} \dots
      \sum_{x_{1}^{\floor{\pi_1N}}\in\Omega_M}
      &\sum_{x_2^{p_2+1}\in \Omega_M} \dots
      \sum_{x_{2}^{\floor{\pi_2N}}\in\Omega_M}\dotsi \dots
      \sum_{x_{n}^{p_n + 1}\in\Omega_M}
      \dots \sum_{x_{n}^{\floor{\pi_nN}}\in\Omega_M} \\
      &\mu^{N}(x_1^1,\dots,x_1^{\floor{\pi_1 N}},x_2^1,\dots,x_2^{\floor{\pi_2
        N}},\dots\dots,x_n^1,\dots,x_n^{\floor{\pi_n N}})
    \end{aligned}
  \end{equation*}
  for
  \begin{equation*}
    \projV{x} := (x_1^1,\dots,x_1^{p_1},x_2^1,\dots,x_2^{p_2},\dots \dots,x_n^1,\dots, x_n^{p_n}).
  \end{equation*}
  For $\mu^N \in \probP(\Omega_M^N)$, we denote the marginal by
  \begin{equation*}
    \mu^{N;(p)} = \projP^{N;(p)} \mu^{N} \in \probP(\Omega_M^{\otimes (p_1+\dots+p_n)}).
  \end{equation*}
\end{definition}

We then expect to recover the master equation from the first marginals
\begin{equation*}
  u_i := \mu^{N;(e_i)},
\end{equation*}
where $e_i = (0,\dots,0,1,0,\dots,0)$ denotes the unit vector with $n$
components, where the $1$ is at the $i$-th component.

The quadratic terms are expected to originate from a binary
interaction in the particle model. The diffusion is the result of a
random walk and the nonlinearity given in the SKT model
\eqref{eq:intro-skt} is expected to come out from jumps of
% interactions with
particles interacting at the same position.

The entropy structure is expected to be linked to the reversibility of
the Markov chain. We therefore introduce a binary interaction which
happens in a reversible way. This can be realised by imposing the same
jump for the interacting particles.

This leads us to consider the following class of particle models,
describing the evolution of the microscopic configuration.

\begin{definition}[Reversible particle model]
  \label{def:markov-chain-particle-model}
  Let $D_i$ and $D_{ij}$ be nonnegative constants such that $D_{ij} = D_{ji}$ for
  $i,j=1,\dots,n$. For a fixed $N$, define the time-continuous Markov
  chain on $\Omega_M^N$ by the transitions
  \begin{align*}
    &\left.
      \begin{aligned}
        \fullV{x} \to \fullV{x} + \fullV{e}_i^a + \fullV{e}_j^b \\
        \fullV{x} \to \fullV{x} - \fullV{e}_i^a - \fullV{e}_j^b
      \end{aligned}
    \right\}
    &\text{ with rate }&
                         \delta_{(i,a)\not=(j,b)}
                         \delta_{x_i^a=x_j^b}
                         \frac{D_{ij}}{N} \\
    &\left.
      \begin{aligned}
        \fullV{x} \to \fullV{x} + \fullV{e}_i^a \\
        \fullV{x} \to \fullV{x} - \fullV{e}_i^a
      \end{aligned}
    \right\}
    &\text{ with rate }&
                         D_i
  \end{align*}
  for $i,j=1,\dots,n$ and $a=1,\dots,\floor{\pi_iN}$, $b=1,\dots,\floor{\pi_jN}$,
  where $\fullV{e}_i^a$ is the vector with components of value zero at all places, except for the $a$th particle of species
  $i$, where the value is $h = 1/M$. The Markov chain is defined to
  have no other transitions.
\end{definition}

\begin{remark}
  The transition rates are well-defined if and only if
  $D_{ij} = D_{ji}$, which will lead to condition
  \eqref{eq:detailed-balance-skt}.
\end{remark}

From the construction, we directly see the reversibility.
\begin{lemma}\label{lemma.2}
  The Markov chain given in \cref{def:markov-chain-particle-model} is
  reversible and the stationary distribution is the homogeneous
  distribution, where each $\fullV{x} \in \Omega^{N}_M$ has the same
  probability $|\Omega^N_M|^{-1} = M^{-(\floor{\pi_1 N} + \ldots+ \floor{\pi_n N})}$.
\end{lemma}
\begin{proof}
  This can be obtained by a direct computation.
\end{proof}

We now suppose that the microscopic configuration evolves according to
the Markov chain. Then the distribution $\mu^N$ solves the following linear ODE:
\begin{equation}
  \label{eq:evolution-density-markov-chain}
  \begin{aligned}
    \frac{\dd}{\dd t} \mu^N(\fullV{x})
    &= \sum_{i=1}^n\sum_{a=1}^{\floor{\pi_iN}}
    D_i
    \Big[\mu^N(\fullV{x}+\fullV{e_i^a})+\mu^N(\fullV{x}-\fullV{e_i^a})-2\mu^N(\fullV{x})\Big] \\
    &+ \frac{1}{2}
    \sum_{i=1}^n \sum_{a=1}^{\floor{\pi_iN}}
    \sum_{j=1}^n \sum_{b=1}^{\floor{\pi_jN}}
    \delta_{(i,a)\not=(j,b)}
    \delta_{x_i^a=x_j^b}
    \frac{D_{ij}}{N}
    \Big[\mu^N(\fullV{x}+\fullV{e_i^a}+ \fullV{e_j^b})+\mu^N(\fullV{x}- \fullV{e_i^a}-\fullV{e_j^b})-2\mu^N(\fullV{x})\Big].
  \end{aligned}
\end{equation}
Here we used that $x_i^a=x_j^b$ holds after the pairwise interaction
if and only if it holds before, so that we can factor it out
(that is, $\delta_{x_i^a=x_j^b} =
\delta_{(\fullV{x}+\fullV{e_i^a})_i^a=(\fullV{x}+\fullV{e_j^b})_j^b} =
\delta_{(\fullV{x}-\fullV{e_i^a})_i^a=(\fullV{x}-\fullV{e_j^b})_j^b}$).

We further suppose that the particles are indistinguishable, which is
propagated in time.
\begin{lemma}[Propagation of indistinguishability]
  Suppose that $\mu^N$ is the distribution for the Markov chain given
  in \cref{def:markov-chain-particle-model}. If $\mu^N \in
  \probP_s(\Omega_M^N)$ initially holds, then it also holds at all later times.
\end{lemma}
\begin{proof}
  It follows directly from the definition of the transition rates, which
  respect the indistinguishability.
\end{proof}

 We now write explicitly the formula emphasizing the entropy structure of our reversible Markov process, we recall that due to \cite{maas-2011-gradient-flows}, the time-reversible many-particle continuous time Markov chain is a gradient flow of the relative entropy with respect to its stationary distribution.

 \begin{lemma}\label{lemma.tilde.H} We assume that $D_i\ge 0$, and
   $D_{ij}=D_{ji} \ge 0$ for all $1\leq i,j\leq n$. We also assume that $\mu^N$ is initially strictly positive.
   Then, the entropy functional defined by
   \begin{align}\label{tilde.H}
     \tilde{\mathcal{H}}(\mu^N) : =\sum_{\fullV{x}}\mu^N(\fullV{x})\log\left(\frac{\mu^N(\fullV{x})}{M^{(\floor{\pi_1 N}+ \cdots + \floor{\pi_n N})}}\right)
   \end{align}
   is decreasing with respect to time, i.e.
   \begin{align*}
     \frac{\dd}{\dd t}\tilde{\mathcal{H}}(\mu^N) \leq 0 \quad
     \text{for all $t>0$}
   \end{align*}
   along the flow of \eqref{eq:evolution-density-markov-chain}.
 \end{lemma}
\begin{proof}
  Note first that the strict positivity of $\mu^N$ is maintained in
  the evolution of the process, so that the logarithm of $\mu^N$ is
  always well defined.

  The proof works in a totally analogous way as the proof of
  \eqref{estim.2}, where the entropy decay is shown on the macroscopic
  level. For completeness, we sketch the proof also here, by using the
  following notation for any function $f: \Omega^N_M \to (0,\infty)$:
  \begin{align*}
    \Delta_{\fullV{e}_i^a}(f(\fullV{x}))&:= f(\fullV{x} + \fullV{e}_i^a) + f(\fullV{x} - \fullV{e}_i^a) - 2f(\fullV{x}), \\
    \Delta_{(\fullV{e}_i^a + \fullV{e}_j^b)}(f(\fullV{x}))&:= f(\fullV{x} + \fullV{e}_i^a + \fullV{e}_j^b) + f(\fullV{x} - \fullV{e}_i^a - \fullV{e}_j^b) - 2f(\fullV{x}), \\
    \nabla^+_{(\fullV{e}_i^a)}(f(\fullV{x}))&:= f(\fullV{x} + \fullV{e}_i^a) - f(\fullV{x}), \\
    \nabla^+_{(\fullV{e}_i^a + \fullV{e}_j^b)}(f(\fullV{x}))&:= f(\fullV{x} + \fullV{e}_i^a + \fullV{e}_j^b)- f(\fullV{x}).
  \end{align*}
  Here $\fullV{e}_i^a$ is the vector with components of value zero at
  all places, except for the $a$th particle of species $i$,
  where the value is $h = 1/M$.  This coincides with the
  notation of a discrete Laplacian and discrete gradient up to
  positive scaling constants. We will introduce them more rigorously
  on the level of the master equation in Section
  \ref{sec:rigorous.derivation}.  Thanks to the periodicity of
  the domain and to a discrete integration by parts (see detailed
  formulas at the beginning of Section \ref{sec:rigorous.derivation}),
  it holds that
  \begin{align*}
    \frac{\dd}{\dd t}\tilde{\mathcal{H}}(\mu^N) &= \sum_{\fullV{x}} \left(\log \mu^N(\fullV{x}) +1 \right)\frac{d}{dt}\left(\mu^N(\fullV{x})\right) \\
                                                & = \sum_{\fullV{x}}\sum_{i=1}^n \sum_{a=1}^{\floor{\pi_i N}}D_i \left(\log \mu^N(\fullV{x}) +1 \right) \Delta_{\fullV{e}_i^a}\left(\mu^N(\fullV{x})\right) \\
                                                & \quad +\frac12\sum_{\fullV{x}}\sum_{i,j=1}^n \sum_{a=1}^{\floor{\pi_i N}}\sum_{b=1}^{\floor{\pi_j N}}
                                                  \delta_{(i,a)\not=(j,b)}
                                                  \delta_{x_i^a=x_j^b}\frac{D_{ij}}{N}\left(\log \mu^N(\fullV{x}) +1 \right)\Delta_{(\fullV{e}_i^a + \fullV{e}_j^b)}\left(\mu^N(\fullV{x})\right) \\
                                                &= -\sum_{\fullV{x}}\sum_{i=1}^n \sum_{a=1}^{\floor{\pi_i N}}D_i \nabla^+_{(\fullV{e}_i^a)}\left(\log \mu^N(\fullV{x})\right)\cdot \nabla^+_{(\fullV{e}_i^a)}\left(\mu^N(\fullV{x})\right) \\
                                                & \quad -\frac12\sum_{\fullV{x}}\sum_{i,j=1}^n \sum_{a=1}^{\floor{\pi_i N}}\sum_{b=1}^{\floor{\pi_j N}}
                                                  \delta_{(i,a)\not=(j,b)}
                                                  \delta_{x_i^a=x_j^b}\frac{D_{ij}}{N}\nabla^+_{(\fullV{e}_i^a + \fullV{e}_j^b)}\left(\log \mu^N(\fullV{x})\right)\cdot \nabla^+_{(\fullV{e}_i^a + \fullV{e}_j^b)}\left(\mu^N(\fullV{x})\right) \\
                                                & \leq 0,
  \end{align*}
   thanks to the monotonicity of $x \mapsto \log x$.
\end{proof}

The evolution of the marginals is given by the BBGKY hierarchy.
\begin{lemma}[BBGKY hierarchy]
  \label{thm:bbgky}
  Suppose that $\mu^N \in \probP_s(\Omega_M^N)$ is the density
  evolving according to \eqref{eq:evolution-density-markov-chain}. Then the
  marginals evolve as
  \begin{equation*}
    \frac{\dd}{\dd t} \mu^{N;(p)}(\projV{x})
    = I + II + III,
  \end{equation*}
  where
  \begin{align*}
    I &=  D_i\sum_{i=1}^n\sum_{a=1}^{p_i}
        [\mu^{N;(p)}(\projV{x}+\projV{e}_i^a)+\mu^{N;(p)}(\projV{x}-\projV{e}_i^a)-2\mu^{N;(p)}(\projV{x})],
    \\
    II &= \frac{1}{2}
         \sum_{i=1}^n \sum_{a=1}^{p_i}
         \sum_{j=1}^n \sum_{b=1}^{p_j}
         \delta_{(i,a)\not=(j,b)}
         \delta_{x_i^a=x_j^b}
         \frac{D_{ij}}{N}
         [\mu^{N;(p)}(\projV{x}+\projV{e}_i^a+\projV{e}_j^b)+\mu^{N;(p)}(\projV{x}-\projV{e}_i^a-\projV{e}_j^b)-2\mu^{N;(p)}(\projV{x})],
  \end{align*}
  with $\projV{e}_i^a$ defined as the vector of size $p_1 + \dots + p_n$ with all coordinates with value
  $0$, except the coordinate of index
  $p_1 + \dots + p_{i-1} +a$, which value is $1/M=h$. Finally,
  \begin{align*}
  III
      &=
        \sum_{i=1}^n \sum_{a=1}^{p_i}
        \sum_{j=1}^n
        \sum_{x_j^{p_j+1} \in \color{blue}{\Omega_M}}
        \delta_{x_i^a=x_j^{p_j+1}} D_{ij} \frac{\floor{\pi_jN}-p_j}{N} \\
      &\quad
        \Big[\mu^{N;(p+e_j)}((\projV{x}\# x_j^{p_j+1})+\tilde{e}_i^a+\tilde{e}_j^{p_j + 1})
        +\mu^{N;(p+e_j)}((\projV{x}\# x_j^{p_j+1})-\tilde{e}_i^a-\tilde{e}_j^{p_j +1})
        -2\mu^{N;(p+e_j)}((\projV{x}\# x_j^{p_j+1}))\Big]
  \end{align*}
  with
  \begin{equation*}
    (\projV{x}\# x_j^{p_j+1})
    = (x_1^1,\dots,x_1^{p_1},\dots\dots,x_j^1,\dots,x_j^{p_j},x_j^{p_j+1},x_{j+1}^1,\dots, x_{j+1}^{p_{j+1}},\dots\dots,x_n^1,\dots,x_n^{p_n}),
  \end{equation*}
  i.e.\ $\projV{x}$ with $x_j^{p_j+1}$ added between $x_j^{p_j}$ and $x_{j+1}^1$, and where $\tilde{e}_i^a$
is defined  as the vector of size $p_1 + \dots + (p_j +1) + \dots + p_n$ with all coordinates with value $0$, except the coordinate of index $p_1 + \dots + p_{j} +1$,
  which value  is $1/M=h$.
\end{lemma}
The term $I$ is the standard linear diffusion. The term $II$ is the
quadratic interaction between the considered particles, which should
be negligible as $N \to \infty$. The term $III$ is the interaction
between the considered particles and the averaged particles, which
leads to the quadratic term. The interaction between the averaged
particles does not appear in the projection.

\begin{proof}
  Take the projection $\projP^{N;(p)}$. The terms $I$ and $II$ follow
  directly. The third term appears as
  \begin{equation*}
    III =
    \sum_{i=1}^n \sum_{a=1}^{p_i}
    \sum_{j=1}^n \sum_{b=p_j+1}^{\floor{\pi_jN}}
    \delta_{x_i^a=x_j^b}
    \frac{D_{ij}}{N}
    \projP^{N;(p)}[\mu^{N}(\fullV{x}+\fullV{e}_i^a+\fullV{e}_j^b)+\mu^{N}(\fullV{x}-\fullV{e}_i^a-\fullV{e}_j^b)-2\mu^{N}(\fullV{x})],
  \end{equation*}
  where we ordered the pair $(i,a)$ and $(j,b)$ so that the factor
  $1/2$ is not appearing there. Thanks to the indistinguishability, this
  takes the claimed form.
\end{proof}

Thus, as usually in the BBGKY hierarchy (tracing back to \cite{Bog46}),  in order to compute the evolution of the one-particle marginals
$u_i$, we need the knowledge of the two-particle marginals, whose
evolution in turn requires the three-particle marginals.

In order to close an equation on $u_i$, we thus need an additional
assumption. This assumption has been identified as chaos by Kac \cite{Kac56}, in
a mathematical setting following the famous Sto{\ss}zahlansatz by
Boltzmann (first suggested by J.~Clerk~Maxwell in  \cite{Max1867}).
It states that in the limit $N \to \infty$, the
different particles are becoming independent. In terms of the measure
$\mu^N$, it means that
\begin{equation}
  \label{eq:statement-chaos}
  \begin{aligned}
    &\mu^{N}(x_1^1,\dots,x_1^{\floor{\pi_1 N}},x_2^1,\dots,x_2^{\floor{\pi_2
      N}},\dots\dots,x_n^1,\dots,x_n^{\floor{\pi_n N}}) \\
    &\approx u_1(x_1^1)\dotsm u_1(x_1^{\floor{\pi_1 N}}) \, u_2(x_2^1)\dotsm u_2(x_2^{\floor{\pi_2 N}})\dotsm\dotsm u_n(x_n^1)\dotsm
    u_n(x_n^{\floor{\pi_nN}}),
  \end{aligned}
\end{equation}
as $N\to \infty$. {\color{blue} The formal idea is that the interaction
between two particles is scaled as $N^{-1}$ so that the correlation
between two particles should also be scaled as $N^{-1}$. Therefore, as
$N\to\infty$, the particles become independent in the limit. In a way
for a rigorous mathematical treatment, Kac suggested to initially
assume the factorisation and then prove that this is preserved in time
with an error going to zero as $N\to\infty$ (propagation of chaos).}

\begin{proposition} \label{six}
  Formally, as $N \to \infty$ we have under the Sto{\ss}zahlansatz
  that the marginals $u_i$ evolve as
  \begin{align} \label{eqsix}
    \frac{\dd}{\dd t} u_i(x) &= D_i \Big[u_i (x+h) + u_i (x-h) - 2 u_i (x)\Big]\\
    &\quad +\sum_{j=1}^{n} D_{i j} \pi_j
    \Big[u_j(x+h)u_i (x+h) + u_j(x-h)u_i (x-h) - 2 u_j(x)u_i (x)\Big].\notag
    \end{align}
\end{proposition}
\begin{proof}
  This follows from \cref{thm:bbgky}, where in the term $III$, we see
  that  $\frac{\floor{\pi_jN}-p_j}{N}\to\pi_j$, and we reduce the two-marginal
  density as a product by \eqref{eq:statement-chaos}.
\end{proof}

This gives the desired quadratic master equation with the final rate
$A_{ij} := D_{ij} \pi_j$. This is equivalent to the detailed balance
equation \eqref{eq:detailed-balance-skt}, which follows from the
symmetry $D_{ij}=D_{ji}$ in the following way:
\begin{align*}
 \pi_i A_{ij} = \pi_i D_{ij} \pi_j = \pi_j D_{ji} \pi_i = \pi_j A_{ji}\quad \mbox{for all}~~1\leq i,j \leq n.
\end{align*}

 Thanks to the chaos assumption, we can relate the relative entropy of $\mu^N$ to the relative entropy of the
 $u_i$ using the following proposition.
\begin{proposition}\label{prop.entropy}
  Assume
  \begin{equation*}
    \begin{aligned}
      &\mu^{N}(x_1^1,\dots,x_1^{\floor{\pi_1 N}},x_2^1,\dots,x_2^{\floor{\pi_2
        N}},\dots\dots,x_n^1,\dots,x_n^{\floor{\pi_n N}}) \\
      &= u_1(x_1^1)\dotsm u_1(x_1^{\floor{\pi_1 N}}) \, u_2(x_2^1)\dotsm u_2(x_2^{\floor{\pi_2 N}})\dotsm\dotsm u_n(x_n^1)\dotsm
      u_n(x_n^{\floor{\pi_nN}}),
    \end{aligned}
  \end{equation*}
  then
  \begin{equation*}
  \frac{1}{N}  \tilde{\mathcal{H}}(\mu^N) = \sum_{\fullV{x}} \mu^N(\fullV{x})
    \log\left( \frac{\mu^N(\fullV{x})}{M^{(\floor{\pi_1 N}+\dots+\floor{\pi_nN})}} \right)
    = \frac1{N}\,\sum_{i=1}^n \sum_{\ell=0}^{M-1} \floor{\pi_i N} u_i(x_{\ell})
    \log \left(\frac{u_i(x_{\ell})}{M} \right)
  \end{equation*}
\end{proposition}
\begin{proof}
  This follows from expanding the logarithm as product and using that
 the  $u_i$ are probability distributions.
\end{proof}

When $N \to \infty$, this last quantity converges towards $\sum_{i=1}^n \pi_i \, \sum_{\ell=0}^{M-1} u_i(x_{\ell})\,
    \log \left(\frac{u_i(x_{\ell})}{M} \right)$.
 Because of Lemma \ref{lemma.tilde.H} and Proposition \ref{six}, we expect that this quantity decreases along the flow of eq. (\ref{eqsix}). We shall indeed prove this in the next section, thus establishing the link between the
entropy structure for eq. (\ref{eqsix}) (and its limit when the discretisation step $h$ tends to $0$) and the  classical
relative entropy of Markov chains.

\section{Rigorous derivation to the cross-diffusion system}\label{sec:rigorous.derivation}

Starting from the Markov chain defined in
Definition \ref{def:markov-chain-particle-model}, we showed in  the last section how performing the mean-field limit on the formal
level leads to the spatial discretisation
(\ref{eqsix}) of the SKT system (\ref{eq:intro-skt}).  In fact, we
expect this particular discretisation to preserve the entropy
structure of the Markov chain. In this section we shall check this
property and use it to pass rigorously to the limit when the
discretisation step $h$ tends to $0$, thus recovering the existence of
weak solutions for the SKT model.

For this, we recall the discretisation
$\Omega_M = \{x_k = kh : k=0,\dots,M-1\}$ with $h=M^{-1}$ from
\eqref{eq:def-omega-m}. Moreover, we introduce the discrete derivatives and
discrete Laplacian by
\begin{align*}
  (\nabla^+_h f) (x) & := \frac{f(x+h) - f(x)}{h},
                  \qquad  (\nabla^-_h f) (x)  := \frac{f(x) - f(x-h)}{h}, \\
  (\Delta_h f)(x) & := [\nabla^-_h(\nabla^+_h f)](x)
                    = [\nabla^+_h(\nabla^-_h f)](x)
                    = \frac{f(x+h) + f(x-h) - 2 f(x)}{h^2}.
\end{align*}
We now rewrite \eqref{eqsix} together with its initial boundary
conditions (and $A_{ij} := \pi_j\, D_{ij}$), after a suitable rescaling in time (such that $\pa_t$ is replaced by $h^2\,\pa_t$). This yields

\begin{equation}\label{ode}
  \left\{
  \begin{lgathered}
    \pa_t u_i(t,x_k) = D_i\,[\Delta_h u_i(t,\cdot)](x_k) + \left[\Delta_h \left(u_i(t,\cdot)\sum_{j=1}^n A_{ij}u_j(t,\cdot)\right)\right] (x_k) , \quad k=0,\dots,M{-}1,\\
    u_i(0,x_k) =u_i^0(x_k)\ge 0, \qquad k=0,\ldots,M{-}1, \quad i=1,\ldots,n,\\
    u_i(t,x_0) = u_i(t,x_M), \quad  u_i(t,x_{-1}) = u_i(t,x_{M-1}), \qquad \forall t, \quad i=1,\ldots,n.
  \end{lgathered}
  \right.
\end{equation}

Given the values $(w(x_k))_{k=0,\dots,M-1}$ over $\Omega_M$, let
$\tilde{w} : \T \mapsto \R$ be the linear interpolant ($P_1$ discretisation), which can be
defined by
\begin{align}\label{tilde.u}
  \tilde{w}(x) := \sum_{k=0}^{M-1} w(x_k)\, T(x - x_k) + w(x_0)\, T(x - x_M),
\end{align}
where
\begin{align}\label{defi.T}
  T(x) = (1 - |x|/h)\, \indicator_{[\,|x|\le h\,]}.
\end{align}
With this we can state our main theorem.
\begin{theorem}
  \label{main}
  Let $D_i \ge 0$ and $A_{ij} \ge 0$ be coefficients satisfying
  \renewcommand{\theenumi}{\roman{enumi}}
  \begin{enumerate}
  \item $A_{ii} > 0$ (strict positivity of self-diffusion),
  \item $\pi_i A_{ij} = \pi_j A_{ji}$ for some constants $\pi_i >0$
    (detailed balance equation).
  \end{enumerate}
  We also assume continuous positive initial data $u_i^0:= u_i^0(x) > 0$ on
  $\mathbb{T}$ for all $1 \leq i \leq n$.

  Then for all $M \in \mathbb{N}$, there exists a unique global
  solution $u_i := u_i(t,x_k) > 0$ of class $C^{\infty}$ to the
  discrete system (\ref{ode}) with $h := 1/M$.

  Denoting $[\tilde{u}_i]^M$ the interpolant obtained from
  $u_i(t,x_k)$ by formula (\ref{tilde.u}), then there exists a
  subsequence such that the following holds:
  $[\tilde{u}_i]^M \to_{M \to \infty} u_i$ in
  $L^{4 -\var}([0,T] \times \mathbb{T})$ for all $T>0$ and $\var >0$,
  where $u_i \in L^4([0,T]\times \T) \cap L^2([0,T], H^1(\T))$ is a
  weak solution to the SKT system
  $\pa_t u_i = \Delta (D_i \, u_i + \sum_{j=1}^n A_{ij} u_i\,u_j)$
  (with initial data $u_i^0$ and periodic boundary conditions), in the
  following sense: For all
  $\varphi \in C^2_c([0,\infty) \times \mathbb{T})$ and for all
  $i=1,\ldots,n$,
  \begin{equation*}
    \begin{split}
      &- \int_{\mathbb{T}} u_i(0,x) \, \varphi(0,x) \,\dd x -
      \int_0^{\infty}\int_{\mathbb{T}} u_i(t,x) \pa_t \varphi(t,x) \, \dd x\,
      \dd t \\
      &= \int_0^{\infty}\int_{\mathbb{T}} \bigg[ D_i\,u_i(t,x) + \sum_{j=1}^n
      A_{ij}\,u_i(t,x)\, u_j(t,x)\bigg]\, \Delta\varphi(t,x) \,\dd x\,
      \dd t.
    \end{split}
  \end{equation*}
\end{theorem}

The existence of weak solutions to the SKT model has been known for a
long time in the case of two equations, and it has been studied more
recently in the case of more than two species, under the detailed
balance condition (cf.\ \cite{CJ06,CDJ18}). We do not go
beyond the existing theory of existence in this paper. We, however,
present an approximation procedure which is extremely simple (it
consists only in discretizing w.r.t. the space variable)
compared to most previous procedures (cf. for
example \cite{DLMT15,DDJ17}).

Though this approximation procedure is presented here only in the
specific case of the quadratic SKT model without reaction terms under
the assumption of detailed balance in dimension $1$ and in the
presence of self-diffusion, our feeling is that it can be easily
extended to more general cases. First, one can introduce (not too quickly increasing) reaction terms. Second, one can go to higher space dimensions $d\geq 1$ keeping periodic boundary conditions. In a third step, by introducing a reasonable grid, one can expect that the same procedure  works for any reasonably smooth domain with Neumann boundary conditions.
Moreover, one can also think of non quadratic cases, provided that a good Lyapunov
functional is known, or of the quadratic case without self-diffusion when the standard
diffusion term or the reaction terms are sufficient to guarantee the
equintegrability. It is less clear if duality arguments (cf. for
example \cite{DLM14}) are compatible with this approximation (as they
are when time discretisation is performed, cf. \cite{LM17}): this
issue will be investigated in future works.

The possibility of extending the formal results of the first part to
more general systems (thus giving a microscopic background for an
entropy structure which is known to exist at the macroscopic level)
will also be studied further, especially in the direction of non
quadratic systems, and systems presenting exclusion processes, see for
instance \cite{BDPS10, BSW12, BHRW16}. \medskip

We now begin the
\medskip

{\bf{Proof of Theorem \ref{main}}}:

We first observe that there exists a unique global solution
$t\in \R_+ \mapsto (u_1(t, x_k),\ldots,u_n(t,x_k))$ with
$u_i(t,x_k)\ge 0$ to the ODE system \eqref{ode}. We briefly sketch the proof of this result, which uses standard theorems for ODEs.
\par
We denote by $T_1>0$ the maximal time of existence for the equation (obtained thanks to Cauchy-Lipschitz theorem), and by $T_2 \in ]0, T_1[$ the maximal time for which  $u_i(t,x_k)>0$ for all $i,k$ and $t\in [0,T_2[$. Note that $T_2>0$ because all initial data are assumed to be strictly positive.
\par
On the interval $[0,T_2[$, we use the conservation of
the total number of individuals (of each species)
\begin{equation*}
  \frac{\dd}{\dd t} \left( \sum_{k=0}^{M-1} u_i(t, x_k)\right) = 0,
\end{equation*}
for $i=1,\dots,n$, and get that
\begin{equation}\label{tt}
 \forall t\in ]0,T_2[,\,\, i=1,..,n,\,\, k=0,..,M-1, \qquad 0 \le  u_i(t, x_k) \le C,
 \end{equation}
where $C : = \sup_{i=1,..,n} \bigg[ \sum_{k=0}^{M-1} u_i(0, x_k) \bigg]$.
\par
Then, we observe that on the interval $[0,T_2[$,
%Moreover, this solution is strictly positive because there exists $C>0$ such that
\begin{align*}
 h^2 \,\partial_t u_i(t,x_k) &=  D_i \left[u_i(t,x_k+h) + u_i(t, x_k-h) - 2 u_i(t, x_k)\right] \\
 &\quad +  \sum_{j=1}^n A_{ij}\left[(u_i u_j)(t,x_k+h) + (u_i u_j)(t, x_k-h) - 2 (u_iu_j)(t, x_k)\right] \\
 &\geq -2 D_i\, u_i(t,x_k) -2 \sum_{j=1}^n A_{ij}(u_iu_j)(t,x_k) \\
 &\geq -2\,(D_i + C  \sum_{j=1}^n A_{ij})\, u_i(t,x_k),
\end{align*}
and consequently
\begin{align*}
    u_i(t,x_k) \geq u_{i}(0,x_k)\exp\bigg(- \frac{2}{h^2} \,\left[D_i + C  \sum_{j=1}^n A_{ij} \right]\, T_2 \bigg)>0 \quad \mbox{for all}~~t \in [0,T_2].
\end{align*}
Then $T_2=T_1$, and finally thanks to estimate eq. (\ref{tt}), $T_1=\infty$.

\subsection{Discrete system}

We start by studying the discrete system and establishing the main
{\it{a priori}} estimates, which are uniform with respect to the
spatial discretisation $M$. Those estimates are a direct consequence
of the entropy structure of our models.

\begin{lemma} \label{thm:discrete-a-priori} Under the same assumptions
  on the coefficients and initial data as in \cref{main}, the
  unique solution to the system \eqref{ode} satisfies the following a
  priori estimates, for some constants $C_T>0$ depending only on $T$,
  the initial data, and the coefficients $\pi_i$, $A_{ij}$ and $D_i$:
  \begin{align}\label{estim.1}
    \sum_{i=1}^{n} \int_0^T \,\,  h\, \sum_{k=0}^{M-1} |(\nabla^+_h u_{i}) (t,x_k)|^2
    \dd t \le C_T,
  \end{align}
  and
  \begin{align}\label{estim.2}
    \frac{\dd}{\dd t} \entropyH^h(u(t,\cdot)) \le 0, \qquad
    \sup_{t \in [0,T]} \entropyH^h(u(t,\cdot))
    \leq C_T,
  \end{align}
  where
  \begin{equation}\label{entropy.h}
   \entropyH^h(u) := \sum_{i=1}^{n} h \sum_{k=0}^{M-1} \pi_i \,\bigg[u_i(x_k) \log(u_i(x_k)) - u_i(x_k) +1 \bigg]\,.
 \end{equation}
\end{lemma}
\begin{remark}\label{dt.H}
 \textnormal{The normalised entropy described at the end of Section 2
 %$\tilde{\mathcal{H}}$ introduced in Lemma \ref{tilde.H}
 differs from $\entropyH^h$ in \eqref{entropy.h} only by terms which are constant with respect to time $t$ (due to mass conservation), thus the entropy dissipation of those terms is $0$.}
\end{remark}

For the proof, we rely on the following elementary properties for the
discrete derivatives:
\renewcommand{\theenumi}{\roman{enumi}}
\begin{enumerate}
\item Discrete integration by parts: For all $1$-periodic
  functions $p,q: \R \to \R$,
 \begin{equation}\label{rete.rules1}
   \sum_{k=0}^{M-1} (\nabla^+_h p) (x_k)\, q(x_k) =
   - \sum_{k=0}^{M-1}  p(x_k)\, (\nabla^-_h q)(x_k) ;
 \end{equation}
\item Discrete product rule: For all functions $p,q: \R \to \R$,
  \begin{equation} \label{rete.rules2}
    (\nabla^+_h(pq))(x) = p(x+h)\, (\nabla^+_h q)(x) +
    (\nabla^+_hp)(x)\, q(x).
  \end{equation}
\end{enumerate}

\begin{proof}
  By using the abbreviation $\tilde{a}_{ij}:=\pi_i A_{ij}$ (and
  therefore assuming that $\tilde{a}_{ij}=\tilde{a}_{ji}$ for all
  $i,j \in \{1,\ldots,n\}$), we can compute
  \begin{align*}
    &\frac{\dd}{\dd t}\entropyH^h(u) = \sum_{i=1}^n \sum_{k=0}^{M-1} h\, \pi_i\, \pa_t u_i(x_k)\, \log u_i(x_k) \\
    &=\sum_{i=1}^n \sum_{k=0}^{M-1} h\pi_i \left(\Delta_h\left[D_i u_i + u_i\sum_{j=1}^n A_{i,j}u_{j}\right]\right)(x_k)\log u_i(x_k) \\
    &= h \sum_{i=1}^n \sum_{k=0}^{M-1}  D_i\pi_i\, \log u_i(x_k)\, (\Delta_h u_{i})(x_k)  + h\sum_{i,j=1}^n \sum_{k=0}^{M-1}  \tilde{a}_{ij}\,\log u_i(x_k)\, (\Delta_h (u_{i}\,u_j))(x_k)\\
    &=  -h\sum_{i=1}^n \sum_{k=0}^{M-1} D_i\pi_i\,( \nabla^+_h (\log u_i))(x_k) \, (\nabla^+_h(u_i))(x_k) -h\sum_{i,j=1}^n \sum_{k=0}^{M-1} \tilde{a}_{ij}( \nabla^+_h (\log u_i))(x_k) \, (\nabla^+_h(u_i\,u_j))(x_k)  \\
    & =  -h\sum_{i=1}^n \sum_{k=0}^{M-1} D_i\pi_i\,( \nabla^+_h (\log u_i))(x_k) \, (\nabla^+_h(u_i))(x_k) - \frac{h}2\, \sum_{i,j=1}^n \sum_{k=0}^{M-1} \tilde{a}_{ij}( \nabla^+_h (\log (u_i\,u_j)))(x_k) \, (\nabla^+_h(u_i\,u_j))(x_k)\\
    &   \le -4h\sum_{i=1}^n \sum_{k=0}^{M-1} D_i\pi_i\, |\nabla^+_h(\sqrt{u_i})(x_k)|^2   - 2h\, \sum_{i,j=1}^n \sum_{k=0}^{M-1} \tilde{a}_{ij} | \nabla^+_h (\sqrt{u_i\,u_j}))(x_k)|^2  \le 0 .
  \end{align*}
  We used above the elementary inequality
  $(x-y)(\log x- \log y)\geq 4 (\sqrt{x}-\sqrt{y})^2$ for all
  $x>0,y>0$.

  We end up the proof of estimate \eqref{estim.2} by noticing that all
  terms above are nonpositive and by integrating between $0$ and
  $T$. Estimate \eqref{estim.1} is obtained by using only the self
  diffusion terms (that is, the ones corresponding to
  $\tilde{a}_{ij}$, for $i=j$) and also by integrating between $0$ and
  $T$.
\end{proof}

Next, we introduce for $1\leq p<\infty$ the discrete norm for
$(w(x_k))_{k=0,\ldots,M-1}$ by defining
\begin{align}\label{disc.norm}
  \|w\|^p_{h,p} := h \sum_{k=0}^{M-1} |w(x_k)|^p.
\end{align}

With the help of the following lemma, we can switch between the
discrete norm and the norm of the continuous linear interpolant
$\tilde{w}$ of $w$ defined in \eqref{tilde.u}:

\begin{lemma}\label{lemma6}
 For $1\leq p < \infty$ and $w(x_k)\geq 0$, $k=0,..,M-1$, it holds that
 \begin{align}\label{norms}
   \|\tilde{w}\|^p_{L^p(\T)} \leq \|w\|^p_{h,p}\leq \frac{p+1}{2}\|\tilde{w}\|^p_{L^p(\T)},
 \end{align}
    \begin{align}\label{norms2}
	\|\nabla \tilde{w} \|^p_{L^p(\T)} = \|\nabla^+_h w\|_{h,p}^p,
 \end{align}
 where $\|w\|_{h,p}$ is the discrete norm defined in
 \eqref{disc.norm}, and $\tilde{w}$ is the linear interpolant defined
 in \eqref{tilde.u}.
\end{lemma}
\begin{remark}
  The factor $2/(p+1)$ is necessary, as  can be seen from the case when
  $w(x_k)=1$ for $k=1$ and $w(x_k)=0$ otherwise.
\end{remark}

\begin{proof}
  Note that the linear interpolant $\tilde{w}$ can also be rewritten as
  \begin{align} \label{alpha}
    \tilde{w}(x)=\sum_{k=0}^{M-1}\bigg(\alpha_k(x)w(x_k) + (1-\alpha_k(x))w(x_{k+1})\bigg)\indicator_{[x_k,x_{k+1})}(x),
  \end{align}
  where $w(x_M) := w(x_0)$ and
  \begin{align*}
    \alpha_k(x)=\frac{x_{k+1}-x}{h}.
  \end{align*}

  For $x \in [x_k,x_{k+1})$ with $k=0,\ldots,M-1$, we know thanks to
  \eqref{alpha} that
  \begin{align*}
    \tilde{w}(x)=\alpha_k(x)w(x_k)+(1-\alpha_k(x))w(x_{k+1}).
  \end{align*}
  Since $x \mapsto x^p$ is convex, we see that
  \begin{align*}
    \left|\tilde{w}(x)\right|^p \leq \alpha_k(x)\, |w(x_k)|^p  + (1-\alpha_k(x))\,|w(x_{k+1})|^p,
  \end{align*}
  so that integrating between $x_k$ and $x_{k+1}$, we get
  \begin{align*}
    \int_{x_{k}}^{x_{k+1}}|\tilde{w}(x)|^p\,\dd x \leq \frac{h}{2}|w(x_k)|^p + \frac{h}{2}|w(x_{k+1})|^p,
  \end{align*}
  which shows the first part of \eqref{norms}.

  In the other direction, we find that
  \begin{align*}
    \int_{x_{k}}^{x_{k+1}}|\tilde{w}(x)|^p\,\dd x &= h \int_{\beta=0}^1 |\beta w(x_k) + (1-\beta)w(x_{k+1})|^p\,\dd\beta \\
                                               &= \frac{h}{p+1}\frac{\left[w(x_{k+1})\right]^{p+1}-\left[w(x_k)\right]^{p+1}}{w(x_{k+1}) - w(x_k)} \\
                                               &\geq \frac{h}{p+1}\bigg[\left[w(x_{k+1})\right]^{p} + \left[w(x_{k})\right]^{p} \bigg],
  \end{align*}
  where we used the elementary inequality
  \begin{align}\label{elementary.inequ}
    \frac{A^{p+1}-B^{p+1}}{A-B} \geq A^p + B^p \quad \mbox{for all}~~A,B \geq 0.
  \end{align}
  This elementary inequality is easily proved (by considering
  $A/B$). This finishes the proof of \eqref{norms}.

  For \eqref{norms2}, we see that for $k=0,\ldots,M-1$,
  \begin{align*}
    \nabla \tilde{w}(x) = \left(\nabla^+_h w\right)(x_k)\quad
    \text{for $x \in (x_k,x_{k+1}).$}
  \end{align*}
  This implies
  \begin{align*}
    \|\nabla \tilde{w}\|^p_{L^p(\T)}
    = \int_{\mathbb{T}}\left|\nabla \tilde{w}(x)\right|^p\,\dd x
    =\sum_{k=0}^{M-1} \int_{x_k}^{x_{k+1}} \left|\nabla
    \tilde{w}(y)\right|^p\,\dd y
    = \sum_{k=0}^{M-1} h \left|\left(\nabla^+_h w\right)(x_k)\right|^p = \|\nabla^+_h w\|_{h,p}^p.
  \end{align*}
\end{proof}

\subsection{Uniform a priori estimates for the linear interpolant}

From now on, when we interpolate functions which depend on $t$, we
systematically write $\tilde{w}(t,x)$ instead of
${\widetilde{w(t,\cdot)}}(x)$. We also use the notation $C_T$ for any
constant depending on the time $T$, on the inital data and the
parameters $\pi$, $A_{i,j}$ and $D_i$ of the problem, but not on the
discretisation parameter $h= 1/M$.

Combining \cref{thm:discrete-a-priori} and \cref{lemma6}, we obtain
for $\tilde{u}_i$ with $i=1,\dots,n$ that
\begin{align}
 \sup_{t \in [0,T]} \int_{\mathbb{T}}\tilde{u}_i(t,x)\,\dd x \leq C_T, \label{apriori.utilde.1}\\
 \int_0^T \int_{\mathbb{T}} \left|\nabla
  \tilde{u}_i(t,x)\right|^2\,\dd x\, \dd t \leq C_T. \label{apriori.utilde.2}
\end{align}
Using the Gagliardo-Nirenberg inequality, this implies for
$p \in [1,4]$ that
\begin{align} \label{nn}
  \int_0^T \int_{\mathbb{T}}|\tilde{u}_i(t,x)|^p \,\dd x\, \dd t \leq C_T.
\end{align}
Note that in the estimate above, when the dimension $d=1$ is replaced by a more
general dimension $d$, the maximal value $4$ of $p$ is replaced by
$2+2/d$.

Indeed, the Gagliardo-Nirenberg interpolation allows to estimate
$\|\tilde{u}_i\|_{L^p(\T)}$ by
$\|\nabla \tilde{u}_i\|_{L^2(\T)}^{\theta}
\|\tilde{u}_i\|_{L^1(\T)}^{1-\theta}$. Choosing $\theta p=2$, means
that $\theta=(2d(p-1))/((d+2)p) \in [0,1]$ so that $p=2+ 2/d$. With
this choice, we find
\begin{align*}
  \|\tilde{u}_i\|^p_{L^p([0,T], L^p(\T))}
  &= \int_0^T \|\tilde{u}_i\|^p_{L^p(\T)}\,\dd t \\
  &\leq C \int_0^T \|\nabla\tilde{u}_i\|^{\theta p}_{L^2(\T)}
    \|\tilde{u}_i\|^{(1-\theta)p}_{L^1(\T)}\,\dd t \\
  &\leq C \|\tilde{u}_i\|^{(1-\theta)p}_{L^{\infty}([0,T],
    L^1(\T))}
    \int_0^T \|\nabla\tilde{u}_i\|^{\theta p}_{L^2(\T)}\,\dd t \\
  &\leq C \|\tilde{u}_i\|_{L^{\infty}([0,T], L^1(\T))}^{(1-\theta)p}\|\nabla\tilde{u}_i\|^{2}_{L^2([0,T] \times\T)}.
\end{align*}

Using \cref{lemma6}, we can relate the estimate back to the discrete
system as
\begin{equation}
  \label{eq:discrete-l-p}
  \int_0^T h \sum_{k=0}^{M-1} \left|u_i(t,x_k)\right|^p\,\dd t \leq C_T,
  \quad \text{for all $p\in [1,4]$.}
\end{equation}

We now show that for all $\phi \in W^{1,\infty}(\T)$,
\begin{align} \label{estim.11}
  \int_0^T \bigg| \pa_t \int_{\T} \tilde{u_i}(t,x)\, \phi(x)\, \dd x
  \bigg| \, \dd t \leq C_T\, \|\phi\|_{W^{1,\infty}(\T)}.
\end{align}

Indeed, performing a discrete integration by parts in $x_k$
(cf. \eqref{rete.rules1}) and performing the translation
$x \mapsto x+h$ inside the integral over $\T$, we get that
\begin{align*}
  &\pa_t \int_{\T} \tilde{u_i}(t,x)\, \phi(x)\, \dd x\\
  &= \sum_{k=0}^{M-1}  \bigg[\Delta_h \Big(D_i \,u_i(t,\cdot) +
    u_i(t,\cdot)\sum_{j=1}^n A_{ij} u_j(t,\cdot)\Big)\bigg] (x_k) \,
    \int_{\T}  T(x - x_k) \, \phi(x)\, \dd x \\
  &=  - \sum_{k=0}^{M-1}  \bigg[\nabla^+_h \Big(D_i \,u_i(t,\cdot) +
    \sum_{j=1}^n A_{ij} \, u_i(t,\cdot)\, u_j(t,\cdot)\Big)\bigg]
    (x_k) \, \int_{\T}  [(\nabla^+_h T)(x - \cdot )](x_k) \, \phi(x)\,
    \dd x \\
  &= \sum_{k=0}^{M-1}  \bigg[\nabla^+_h \Big(D_i \,u_i(t,\cdot) +  \sum_{j=1}^n A_{ij} \, u_i(t,\cdot)\, u_j(t,\cdot)\Big)\bigg] (x_k) \, \int_{\T}  \frac{T(x-x_k) - T(x-x_{k+1})}{h} \, \phi(x)\, \dd x \\
  &= - \sum_{k=0}^{M-1} \bigg[\nabla^+_h \Big(D_i \,u_i(t,\cdot) +  \sum_{j=1}^n A_{ij} \, u_i(t,\cdot)\, u_j(t,\cdot)\Big)\bigg] (x_k) \, \int_{\T}  T(x-x_k) \, [\nabla^+_h\phi](x)\, \dd x.
\end{align*}
By the discrete product rule \eqref{rete.rules2}, this can be
estimated as
\begin{align*}
  &\int_0^T \bigg| \pa_t \int_{\T} \tilde{u_i}(t,x)\, \phi(x)\, \dd x \bigg| \, \dd t \\
  & \le   \int_0^T  \sum_{k=0}^{M-1} \bigg| \bigg[\nabla^+_h \Big(D_i \,u_i(t,\cdot) +  \sum_{j=1}^n A_{ij} \,u_i(t,\cdot)\, u_j(t,\cdot)\Big)\bigg] (x_k) \bigg| \,
    \int_{\T}  T(x-x_k) \, |[\nabla^+_h\phi](x)|\, \dd x \, \dd t \\
  & \le \| \nabla^+_h\phi \|_{\infty} \,h\,
    \int_0^T  \sum_{k=0}^{M-1} \bigg| \bigg[\nabla^+_h \Big(D_i \,u_i(t,\cdot) +  \sum_{j=1}^n A_{ij} \, u_i(t,\cdot)\, u_j(t,\cdot)\Big) \bigg] (x_k) \bigg| \, \dd t \\
  & \le \| \nabla^+_h\phi \|_{\infty} \,
    \int_0^T   h\, \sum_{k=0}^{M-1}\sum_{j=1}^n A_{ij}  |u_i(t,x_k+h)|\,
    | [\nabla^+_h u_j(t,\cdot)](x_k) |\, \dd t \\
  & \quad + \| \nabla^+_h\phi \|_{\infty} \,
    \int_0^T   h\, \sum_{k=0}^{M-1}\sum_{j=1}^n A_{ij}  |u_j(t,x_k)|\,
    | [\nabla^+_h u_i(t,\cdot)](x_k) |\, \dd t \\
  & \quad + \| \nabla^+_h\phi \|_{\infty} \,
    \int_0^T   h\, \sum_{k=0}^{M-1} D_i \,
    | [\nabla^+_h u_i(t,\cdot)](x_k) |\, \dd t \\
  & \le \| \nabla^+_h\phi \|_{\infty} \sum_{j=1}^n A_{ij}
    \bigg( \int_0^T   h\, \sum_{k=0}^{M-1}  |u_i(t,x_k+h)|^2\, \dd t \bigg)^{1/2}\,\,
    \bigg( \int_0^T   h\, \sum_{k=0}^{M-1}| [\nabla^+_h u_j(t,\cdot)](x_k) |^2\, \dd t \bigg)^{1/2} \\
  & \quad + \| \nabla^+_h\phi \|_{\infty} \sum_{j=1}^n A_{ij}
    \bigg( \int_0^T   h\, \sum_{k=0}^{M-1}  |u_j(t,x_k)|^2\, \dd t \bigg)^{1/2}\,\,
    \bigg( \int_0^T   h\, \sum_{k=0}^{M-1}| [\nabla^+_h u_i(t,\cdot)](x_k) |^2\, \dd t \bigg)^{1/2} \\
  & \quad + \| \nabla^+_h\phi \|_{\infty} D_i\,
    \bigg( \int_0^T   h\, ( \sum_{k=0}^{M-1} 1 )\,\dd t \bigg)^{1/2}\,\,
    \bigg( \int_0^T   h\, \sum_{k=0}^{M-1}| [\nabla^+_h u_i(t,\cdot)](x_k) |^2\, \dd t \bigg)^{1/2} \\
  &\le C_T \, [n \max_{i,j} A_{ij} + \max_i D_i]\, \|\phi\|_{W^{1,\infty}(\T)}\, \bigg[\,  T^{1/2} + \max_j\,  \bigg( \int_0^T   h\, \sum_{k=0}^{M-1}  |u_j(t,x_k)|^2\, \dd t \bigg)^{1/2} \, \bigg] \\&\le C_T,
\end{align*}
where we used estimates \eqref{estim.1} and \eqref{eq:discrete-l-p}.

\subsection{Compactness}

In order to stress the dependence w.r.t. the spatial discretisation, we denote by
$[\tilde{u}_i]^M$ the interpolant associated to the discrete system on
$\Omega_M$.

The classical Aubin-Lions lemma shows with the estimates
\eqref{apriori.utilde.2}, \eqref{nn} and \eqref{estim.11} that there
exists a subsequence such that
\begin{align*}
  [\tilde{u}_i]^M \to_{M \to \infty} u_i \quad \text{strongly in $L^{4-\var}([0,T] \times\T)$}
\end{align*}
for some $u_i \in L^4([0,T]\times \T) \cap L^2([0,T], H^1(\T))$.

\subsection{Passing to the limit}

We now show that the limit is a solution in the weak formulation
stated in \cref{main}.

We first find that the interpolation $[\tilde{u}_i]^M$ of the discrete
solution on $\Omega_M$ satisfies for all test functions $\varphi :=
\varphi(t,x) \in C_c([0,+\infty)\times \T)$ and $i=1,\dots,N$ that
\begin{equation}\label{13}
  \begin{aligned}
    &-  \int_{\T}[\tilde{u}_i]^M(0,x) \, \varphi(0,x)\,\dd x -
    \int_0^{\infty}\int_{\T}[\tilde{u}_i]^M(t,x)\, \pa_t \varphi(t,x)\,\dd x\, \dd t \\
    & = D_i \int_0^{\infty}\int_{\T} [\tilde{u_i}]^M(t,x)\,
    \Delta_h \varphi(t,x)\,\dd x\, \dd t + \sum_{j=1}^n A_{ij}
    \int_0^{\infty}\int_{\T}[\widetilde{u_iu_j}]^M(t,x)
    \, \Delta_h \varphi(t,x)\,\dd x\, \dd t,
  \end{aligned}
\end{equation}
where $[\widetilde{{u}_iu_j}]^M$ is the interpolant (in the sense of
\eqref{tilde.u}) of ${u}_iu_j$ from the values on $\Omega_M$.

Indeed, differentiating the interpolant $[\tilde{u}]^M$ in time, we
find from the ODE system \eqref{ode} that
\begin{align}\label{equ.lin.interp.}
  \pa_t [\tilde{u}_i]^M(t,x)
  = \sum_{k=0}^{M-1}
  \left[\Delta_h\left( D_i\,u_i(t,\cdot) + \sum_{j=1}^n A_{ij} u_i(t,\cdot)\, u_j(t,\cdot)\right)\right](x_k)\, T(x-x_k).
\end{align}
Multiplying with the compact test function $\varphi$ and integrating shows that
\begin{align*}
  &-  \int_{\T}[\tilde{u}_i]^M(0,x) \, \varphi(0,x)\,dx-\int_0^{\infty} \int_{\T}[\tilde{u}_i]^M\pa_t \varphi \,\dd x\,\dd t \\
  &= \int_0^{\infty} \int_{\T}\sum_{k=0}^{M-1}  \left[\Delta_h\left(D_i\,u_i(t,\cdot) + \sum_{j=1}^n A_{ij} u_i(t,\cdot)\, u_j(t,\cdot)\right)\right](x_k)T(x-x_k)\varphi(t,x)\,\dd x\,\dd t\\
  &=\int_0^{\infty} \int_{\T}\sum_{k=0}^{M-1} \bigg[D_i\,u_i(t,\cdot) + \sum_{j=1}^n A_{ij} u_i(t,\cdot)\,u_j(t,\cdot)\bigg](x_k)\Delta_h\left[T(x- \cdot)\right](x_k)\,\varphi(t,x)\,\dd x\,\dd t\\
  &=\int_0^{\infty} \sum_{k=0}^{M-1} \bigg[D_i\,u_i(t,\cdot) + \sum_{j=1}^n A_{ij} u_i(t,\cdot)\,u_j(t,\cdot)\bigg]
    (x_k)\int_{\T}T(x-x_k)(\Delta_h\varphi)(x)\,\dd x\,\dd t\\
  &= D_i \int_0^{\infty}\int_{\T} [\tilde{u_i}]^M(t,x)\Delta_h \varphi(t,x)\,\dd x\,\dd t + \sum_{j=1}^n A_{ij} \int_0^T\int_{\T}[\widetilde{u_iu_j}]^M(t,x)\Delta_h \varphi(t,x)\,\dd x\,\dd t,
\end{align*}
so that \eqref{13} holds.
\medskip

Next, we use the following result explaining how the linear interpolation behaves on products:

\begin{lemma}\label{lemma.tilde}
  Under the assumptions of \cref{main}, the following estimate holds:
  \begin{align*}
    \int_0^T \int_{\T}\left|[\widetilde{u_iu_j}]^M(t,x)-[\tilde{u}_i]^M(t,x)[\tilde{u}_j]^M(t,x)\right|\,\dd x\,\dd t \leq C_T \, h.
  \end{align*}
\end{lemma}
\begin{proof}
  For $x \in [x_k,x_{k+1})$ the representation formula \eqref{alpha}
  shows for $i=1,\dots,n$ that
  \begin{align*}
    [\tilde{u}]^M_i(t,x)& = \alpha_k(x) u_i(t,x_k) + (1-\alpha_k(x))u_i(t,x_{k+1}),\\
    [\widetilde{u_iu_j}]^M(t,x)& = \alpha_k(x) \left[\left(u_iu_j\right)(t,x_k)\right] + (1-\alpha_k(x))\left[\left(u_iu_j\right)(t,x_{k+1})\right],
  \end{align*}
  where we recall that $\alpha_k(x)=(x_{k+1}-x)/h.$ Then, for
  $x \in [x_k,x_{k+1})$:
  \begin{align*}
    &\bigg|[\widetilde{u_iu_j}]^M(t,x) -
      [\tilde{u}_i]^M(t,x)[\tilde{u}_j]^M(t,x)\bigg| \\
    &=\bigg|\alpha_k(x) \left[\left(u_iu_j\right)(t,x_k)\right] + (1-\alpha_k(x))\left[\left(u_iu_j\right)(t,x_{k+1})\right] \\
    &\quad -\bigg( \left[\alpha_k(x)u_i(t,x_k) +  (1-\alpha_k(x))u_i(t,x_{k+1}) \right]\,\left[\alpha_k(x) u_j(t,x_k) + (1-\alpha_k(x))u_j(t,x_{k+1})\right]\bigg)\bigg|\\
    &\le \alpha_k(x)(1-\alpha_k(x))\bigg(u_i(t,x_k) + u_i(t,x_{k+1}) \bigg) \bigg|u_j(t,x_{k+1}) - u_j(t,x_k)\bigg|.
  \end{align*}
  Consequently, we get that
  \begin{align*}
    &\int_0^T \int_{\T}\left|[\widetilde{u_iu_j}]^M-[\tilde{u}_i]^M[\tilde{u}_j]^M \right|\,\dd x \dd t \\
    &\leq\int_0^T \left(\sum_{k=0}^{M-1}
      \int_{x_k}^{x_{k+1}}\alpha_k(x)(1-\alpha_k(x))\,\dd x\,
      \big(|u_i(t,x_k)| + |u_i(t,x_{k+1})|\big) \big|u_j(t,x_{k+1}) - u_j(t,x_k)\big| \right)\,\dd t \\
    &=\frac16\int_0^T\sum_{k=0}^{M-1} h \,\big(|u_i(t,x_k)| + |u_i(t,x_{k+1})|\big) \, \big|u_j(t,x_{k+1}) - u_j(t,x_k)\big|\,\dd t \\
    &\leq \frac13  \int_0^T \left(\sum_{k=0}^{M-1} h \left|u_i(t,x_k)\right|^2\right)^{1/2}\left(\sum_{k=0}^{M-1} h \left|u_j(t,x_{k+1})-u_j(t,x_k)\right|^2\right)^{1/2}\,\dd t\\
    &=\frac13\, h \left(\int_0^T \sum_{k=0}^{M-1} h \left|u_i(t,x_k)\right|^2\,\dd t\right)^{1/2}\left(\int_0^T\sum_{k=0}^{M-1} h \left|(\nabla^+_h u_j(t,\cdot))(x_k)\right|^2\,\dd t \right)^{1/2}.
  \end{align*}
  This conclude the proof with the estimates \eqref{estim.1} and
  \eqref{eq:discrete-l-p}.
\end{proof}

Thanks to the lemma above, the weak formulation \eqref{13} of the
discretised system implies that for all
$\varphi :=\varphi(t,x) \in C_c([0,T)\times \T)$ and $i\in 1,\ldots,n$,
\begin{align} \label{nn2}
  - \int_{\T} [\tilde{u}_i]^M(0,\cdot)\,\varphi(0,\cdot)\, \dd x
  - \int_0^{\infty} \int_{\T} \left([\tilde{u}_i]^M\,\pa_t \varphi + \bigg[ D_i \,\color{blue}{[\tilde{u}_i]^M} + \sum_{j=1}^n A_{ij} [\tilde{u}_i]^M [\tilde{u}_j]^M \bigg]\, \Delta_h \varphi\right)\,\dd x\, \dd t = \mathcal{O}(h).
\end{align}

We end up the proof of \cref{main} by observing that the limit
satisfies
\begin{align*}
  - \int_{\T} u_i(0,\cdot)\,\varphi(0,\cdot)\, \dd x
  - \int_0^{\infty}\int_{\T} \left(u_i\pa_t \varphi + [D_i\, u_i + \sum_{j=1}^n A_{ij}u_iu_j]\, \Delta \varphi\right)\,\dd x\, \dd t = 0.
\end{align*}
Indeed,
\begin{align*}
  &\left| - 	\int_{\T} u_i(0,\cdot)\,\varphi(0,\cdot) - \int_0^{\infty}\int_{\T} \left(u_i\pa_t \varphi + [D_i \, u_i + \sum_{j=1}^n A_{ij}u_iu_j]\, \Delta \varphi\right)\,\dd x\, \dd t\right|\\
  &\leq  \int_{\T}\left|u_i(0,\cdot)-[\tilde{u}_i]^M(0,
    \cdot)\right|\, |\varphi(0, \cdot)|\,dx
    + \int_0^{\infty}\int_{\T}\left|u_i-[\tilde{u}_i]^M\right||\pa_t\varphi|\,\dd x\, \dd t \\
%  & + \int_0^{\infty}\int_{\T}\left|-[\tilde{u}_i]^M \pa_t \varphi -\Big[D_i [\tilde{u}_i]^M % + \sum_{j=1}^nA_{ij}[\tilde{u}_i]^M[\tilde{u}_j]^M\Big]\Delta_h \varphi\right|\,\dd x\, \dd % t \\
  &+\sum_{j=1}^n A_{ij}\int_0^{\infty}\int_{\T} \left|[\tilde{u}_i]^M [\tilde{u}_j]^M-u_iu_j\right| \, |\Delta_h \varphi| \,\dd x\, \dd t
   +\color{blue}{D_i\,\int_0^{\infty}\int_{\T} \left|[\tilde u_i]^M-u_i\right||\Delta_h\varphi| \,\dd x\, \dd t} \\
    &+\int_0^{\infty}\int_{\T}| D_i\,u_i + \sum_{j=1}^n A_{ij}\,u_iu_j|\left|\Delta_h\varphi - \Delta\varphi\right|\,\dd x\, \dd t +  \color{blue}{\mathcal{O}(h)}.
\end{align*}
The first integral tends to $0$ because $u_i(0,\cdot)$ is continuous
on $\T$. The second and fourth integrals converge to $0$ because
$[\tilde{u_i}]^M\to u_i$ strongly in $L^1$,
%the third integral converges to $0$ because of estimate \eqref{nn2},
 the third integral converges
to $0$ because $[\tilde{u_i}]^M\to u_i$ strongly in $L^2$, and the last
integral converges to $0$ since $\varphi$ is smooth.
\medskip

This concludes the proof of \cref{main}.

\bibliographystyle{elsarticle-num}
\bibliography{lit}

\end{document}